\documentclass[twocolumn, 10 pt]{article}
\usepackage[labelfont=bf]{caption}
\usepackage{blkarray}
\usepackage{amsmath,amsfonts,amsthm,amssymb}
\usepackage{mathtools}
\usepackage{enumitem}
\newtheorem{theorem}{Theorem}[section]
\newtheorem{lemma}[theorem]{Lemma}

\newtheorem{definition}{Definition}

\newtheorem{example}{Example}
\newtheorem{proposition}{Proposition}
\newtheorem{remark}{Remark}
\usepackage{graphicx}
\usepackage{url}
\usepackage{cite}
\usepackage{amsmath,amssymb,amsfonts}
\usepackage{algorithmic}
\usepackage{graphicx}
\usepackage{textcomp}
\usepackage{xcolor}
\usepackage{bm}
\usepackage{cleveref}
\newcommand{\x}{\bm{x}}

\newcommand{\A}{\mathtt{A}}
\newcommand{\B}{\mathtt{B}}

\newcommand{\N}{\mathcal{N}}
\newcommand{\E}{\mathcal{E}}
\newcommand{\G}{\mathcal{G}}
\renewcommand{\x}{\bm{x}}
\newcommand{\tb}{\textcolor{blue}}

\title{Equilibration of Coordinating Imitation and Best-Response Dynamics
}

\author{Nazanin Hasheminejad and Pouria Ramazi}
\begin{document}
\maketitle
\begin{abstract}
Decision-making individuals are often considered to be either \emph{imitators} who copy the action of their most successful neighbors or \emph{best-responders} who maximize their benefit against the current actions of their neighbors. 
In the context of \emph{coordination games}, where neighboring individuals earn more if they take the same action, by means of potential functions, it was shown that populations of all imitators and populations of all best-responders equilibrate in finite time when they become active to update their decisions sequentially. 
However, for mixed populations of the two, the equilibration was shown only for specific activation sequences. 
It is therefore, unknown, whether a potential function also exists for mixed populations or if there actually exists a counter example where an activation sequence prevents equilibration. 
We show that in a \emph{linear graph}, the number of ``sections'' (a sequence of consecutive individuals taking the same action) serves as a potential function, leading to equilibration, and that this result can be extended to \emph{sparse trees}. 
The existence of a potential function for other types of networks remains an open problem. 
\end{abstract}

\emph{Keywords.}
 Decision making,
 best-response, 
 imitation, 
 coordination game,
 convergence.

\section{Introduction}
Evolutionary game theory has been successfully applied in different applications ranging from cancer and epidemiology to finance and rumour propagation \cite{archetti2019cooperation,li2017new,coninx2018gets,askarizadeh2019evolutionary}.
In the context of decision-making, individuals are modeled as game-playing agents who choose from a number of available strategies and accordingly earn payoffs against their matched opponents.  
The agents revise their decisions according to some update rules, the most common being \emph{(myopic) best-response} and \emph{imitation}.
An agent following best-response, called a \emph{best-responder}, chooses the strategy that maximizes its payoff against its neighbors given that they would not change their strategies. 
On the other hand, an agent following imitation, known as an \emph{imitator}, simply imitates a neighbor with a higher payoff. 
The wide use of best response by human has been confirmed in experimental studies \cite{mas2016behavioral}.
Similarly, imitation behavior emerges in several real-world scenarios, 
such as employees' ``costumer sweethearting'' \cite{ertz2022imitation}, building cultural intelligence \cite{pauluzzo2021imitation}, and training language models \cite{srivastava2022beyond}. 

Researchers have explored the existence and convergence towards an equilibrium point in both imitation and best-response dynamics \cite{como2020imitation,fu2020evolutionary, farahbakhsh2021best, hu2019stability}.
In the anti-coordination context where the highest-earning decision is the opposite of the opponent’s, a population of best-responders converges to an equilibrium state \cite{ramazi2016networks}. 
The same holds for a population of best-responders in the coordination context, where the highest-earning strategy matches the opponent's \cite{ramazi2016networks, ramazi2020convergence}.
For populations of imitators, however, equilibration is guaranteed only in the coordination context \cite[Theorem 1]{ramazi2022lower}.
All of these studies used a potential function to prove equilibrium convergence.
Clearly, a mixed population of imitators and best-responders may not equilibrate and undergo perpetual fluctuations. 
The outcome is known for the anti-coordination case: equilibration can take place if and only if there exists an equilibrium \cite{le2020heterogeneous}.
What about a mixed population of imitators and best-responders in the coordination context? 
The existence of an activation sequence was established in \cite{sakhaei2021equilibration} that would drive any such mixed population to an equilibrium state. 
It however remains open whether a potential function exists for such populations, or if there is a counter example where an activation sequence can prevent a mixed population from equilibration.


We start tackling this problem for the simple \emph{linear graph} and find that the number of the so-called ``sections'' (consecutive same-strategy playing agents) serves as a potential function. 
We then extend the results to a ring. 
Next, we proceed to a \emph{starlike} graph, a central ``branching node'' connected to several linear graphs or ``branches''.
We show that there always exists a branch where the number of sections in that branch will again be a potential function, establishing equilibration. 
Finally, we generalize the idea to \emph{sparse trees}, i.e., trees where the distance between each two branching nodes is at least three. 


\section{Model} \label{sec_model}
Consider an undirected network $\G$ over a finite set $\N=\{1,2,\ldots,n\}$ of agents who decide between strategies $\A$ and $\B$ over time $t=0,1,2,\ldots$.
For each agent $i\in\N$, the network defines a set of \emph{neighbors} $\N_i\subseteq\N\setminus\{i\}$ that are connected to agent $i$.
At every time step, each agent $i\in\N$ plays a two-player \emph{(row-column) coordination game} with each of its neighbors $j\in\N_i$ and earns a payoff according to their strategies and its payoff matrix
\begin{equation} \label{coordinationPayoffMatrix}
    \bm{\pi}^i 
    = \begin{bmatrix} R_i & S_i \\ T_i & P_i \end{bmatrix}, \ 
    \min\{R_i, P_i\} > \max\{T_i, S_i\} 
\end{equation}
where $R_i, S_i, T_i$, and $P_i$ are agent $i$'s payoffs when agents $i$ and $j$ play strategy pairs $(\A,\A)$, $(\A,\B)$, $(\B,\A)$, and $(\B,\B)$.
Then agent $i$'s \emph{utility} $u_i$ is the accumulated payoff earned against all of its neighbors:
$
    u_i(\x) = \sum_{j\in\N_i} \bm{\pi}^i_{x_i, x_j}
$ 
where $x_k$ is the strategy of agent $k$, 
the \emph{state} $\x = [x_k]$ is the vector of all agents' strategies, and
$\bm{X}_{pq}$ denotes the entry of matrix $\bm{X}$ at row $p$ and column $q$.
Agents update their strategies based on the type of \emph{update rule} they follow, which is either \emph{best response}, that is to choose the strategy that maximizes its utility, or \emph{imitation}, that is to copy the strategy of its highest earning neighbor. 
The updates happen asynchronously over time, i.e., at each time step, a single agent becomes active to update its strategy at the next time step.
More specifically, agent $i$ active at time $t$ updates its strategy at time $t+1$ to the following if it is an \emph{imitator}:
\begin{equation}    \label{imitatorsUpdateRule}
        x_i(t+1) 
        = x_k(t), \qquad  k=\arg\max_{j\in\N_i} u_j(\x(t)).
\end{equation}
and to the following if it is a \emph{best-responder}:
\begin{equation}\label{bestRespondersUpdateRule}
    x_{i}(t+1) 
    = \arg\max_{\mathtt{X}\in\{\A,\B\}}u_i(\x_{i=\mathtt{X}}(t))
\end{equation}
where $\x_{i=\mathtt{X}}$ is the vector $\x$ where the $i^{\text{th}}$ entry is fixed to strategy $\mathtt{X}$.
In the case where both strategies $\A$ and $\B$ maximize the utilities in  \eqref{imitatorsUpdateRule} or \eqref{bestRespondersUpdateRule}, agent $i$ does not switch strategies, i.e., $x_i(t+1) = x_i(t)$.
\begin{remark}
    The standard inequalities in a coordination game are $R_i>T_i$ and $P_i>S_i$ \cite{riehl2018survey}, implying that player $i$'s payoff is maximized when playing the same strategy as that of its opponent. 
    What condition \eqref{coordinationPayoffMatrix} additionally imposes are the inequalities $P_i>T_i$ and $R_i>S_i$, resulting in the so-called \emph{opponent-coordination} payoff matrix \cite{ramazi2022lower}. 
    Then agent $i$'s payoff increases if its neighbor switches her strategy to that of agent $i$, which proves useful in constructing energy functions for imitation dynamics. 
\end{remark}
\begin{example}\textbf{[Programming languages]}
    Given the required effort to master a new programming language, programmers have to decide between two options each time they program an application: \emph{(i)} the comfort of working in the already experienced language and \emph{(ii)} the benefit of learning a new language.
    Some base their decisions on the prevalence of the language, because common languages are supported by a community of peers who can smoothen the learning experience via online forums. 
    Others may focus on how successful other programmers were in terms of, e.g., their salaries or reputation of developed applications.
    The agents here are the community of App developers who interact via online networks.
    The programming languages  \texttt{Python} and \texttt{Java} may be considered as the strategies
    and a programmer would earn more from his peers if they use the same language.
\end{example}
\begin{example}\textbf{[Social media]} 
    \texttt{Telegram} and \texttt{WhatsApp} are two social media applications. 
    Individuals choosing one of them as their main communication stream may decide based on the (weighted) frequency or satisfaction of their friends on each platform, implying the best response and imitation update rules respectively.
    The individuals also have personal preferences over the apps because of their features, resulting in different payoff matrices.
\end{example}

Define the agents' \emph{activation sequence} as the sequence $\langle a_t\rangle_{t=0}^{\infty}$, where $a_t$ is the active agent at time $t$.
The activation sequence together with update rules \eqref{imitatorsUpdateRule} and \eqref{bestRespondersUpdateRule} govern the state $\x(t)$ and define the \emph{decision-making dynamics}, which we refer to as the \emph{coordinating best-response and imitation dynamics}. 
A state $\x^*\in\{\A,\B\}^{n}$ is an \emph{equilibrium} of the dynamics if under every activation sequence,
$
    \x(0)=\x^* 
$    
implies
$
    \x(t) = \x^*
$
for all $t\geq 0$.
We are interested in determining whether the dynamics eventually equilibrate.
We avoid trivial cases where the dynamics ``get stuck'' at a non-equilibrium state because one or more unsatisfied agents do not get the chance to become active.
To this end, we assume the activation sequence is \emph{persistent}, i.e., \tb{each agent becomes active infinitely many times} \cite{ramazi2017asynchronous}.

It follows from the coordination condition \eqref{coordinationPayoffMatrix} and best-response update rule \eqref{bestRespondersUpdateRule} that if agent $i$ tends to play $\A$ at some state, so does it at any other state with more $\A$-playing neighbors. 
In a more restrictive sense, it can be also shown that if an imitator tends to play $\A$ at some state, so does it at any other state where all of its $\A$-playing neighbors still play $\A$. 
This property is referred to as \emph{$\A$-coordinating} \cite[Definition 2]{sakhaei2021equilibration}, based on which, the existence of an activation sequence that would drive the dynamics from a given initial condition to an equilibrium state was shown in \cite[Lemma 1, Theorem 2]{sakhaei2021equilibration}.
However, it remains open whether the dynamics equilibrate under an arbitrary persistent activation sequence. 

\section{Equilibration results}
The main result of this paper is about the equilibration of ``sparse trees'' as presented in the following theorem.
The \emph{distance} of two nodes in a graph is the number of edges in the shortest path connecting the two.
\tb{A \emph{tree} is a network where exactly one path connects every pair of nodes.}
Define a \emph{branching agent} as an agent with more than two neighbors.
We call a tree network \emph{sparse} if the distance between every pair of its branching agents is greater than two. 
\begin{theorem}[Sparse tree] \label{th_sparseThree}
    A sparse-tree network equilibrates under the coordinating best-response and imitation dynamics with an arbitrary persistent activation sequence.
\end{theorem}
Sparse trees are a generalization of starlikes, which in turn are a generalization of linear graphs, defined in what follows. 
We accordingly, first show the result for linear graphs (as well as rings), then starlikes, and finally sparse trees. 
\subsection{Linear graphs}
Consider network $\G = (\N,\E)$ with edge set $\E=\{\{i,i+1\}\mid i=1,\ldots,n-1\}$, called a \emph{linear graph}.
\begin{definition}[Border agent]
    Given a linear graph, agent $i$ is a \emph{right-border} (resp. \emph{left-border}) if it has a different strategy compared to agent $i+1$ (resp. $i-1$).
    A single agent with a strategy different from those of its two neighbors is both a right and left border agent.
    An agent is a \emph{border} if it is right or left-border (or both).
\end{definition}
We consider the most ``left'' (resp. ``right'') agent, i.e., agent 1 (resp. $n$), as a left (resp. right) border agent.
We can now define the notion of ``section'' as follows (\Cref{fig:line1}).
\begin{figure}[h]
    \centering
    \includegraphics[width=0.4\textwidth]{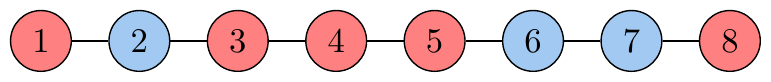}
    \caption{\small
    \textbf{A linear graph with five sections.}
    The sections in this linear graph are $\{1\}$, $\{2\}$, $\{3,4,5\}$, $\{6,7\}$, and $\{8\}$. 
    Blue and red are used for strategies $\A$ and $\B$, respectively. Agents $1$, $2$, and $8$ are each both left and right borders.
    Agents $3$, and $6$ are only left-borders, while agents $5$ and  $7$ are right-borders. 
    Agent $4$ is a non-border agent.
    }
    \label{fig:line1}
\end{figure}
\begin{definition}[Section]
    A \emph{section} in a linear graph at a given strategy state is a set of consecutive same-strategy playing agents $p,p+1,\ldots,q$, where $q\geq p$ and agents $p$ and $q$ are borders.
    The size of the section is defined as $q-p+1$.
\end{definition}
The special case of $p=q$ results in a size-one section consisting of a single agent. 
The number of sections appears to serve as a potential function according to the following lemma. 
The key idea of the proof is that the emergence of a new section requires the sequence $(\A,\A,\A)$ (resp. $(\B,\B,\B)$) to turn into $(\A,\B,\A)$ (resp. $(\B,\A,\B)$), which is impossible due to the coordinating nature of the population dynamics.  
\begin{lemma}\label{lem_numberOfSectionsDoesNotIncrease}
    The number of sections in a linear graph does not increase under the coordinating best-response and imitation dynamics with an arbitrary activation sequence.
\end{lemma}
\begin{proof}
    A change in the population state takes place only if a border agent is active because other agents play the same strategy as their neighbors and hence do not switch strategies according to update rules \eqref{bestRespondersUpdateRule} and \eqref{imitatorsUpdateRule}.
    So the number of sections change at time $t$ only if some border agent $i$ becomes active at time $t-1$ and switches its strategy at time $t$ to say strategy $s$. 
    At least one neighbor of the border agent plays $s$ at time $t-1$ as otherwise, the agent is not border. 
    We have the following two cases, in neither of which the number of sections increases:
    
    \emph{Case 1. Agent $i$ has two neighbors, i.e., $i\not\in\{1,n\}$.}  
    If both neighbors play $s$, then the border agent itself forms a section at time $t-1$, which disappears at time $t$.
    Since no other sections are generated, this results in a reduction in the number of sections. 
    If only one neighbor plays $s$, then the number of sections does not change after the switch.
    
    \emph{Case 2. Agent $i$ has one neighbor, i.e., $i\in\{1,n\}$.} 
    Then the neighbor plays $s$ at time $t-1$, implying that agent $i$ itself again forms a section, which disappears at time $t$, resulting in a reduction.
%
\end{proof}    
As the number of sections are finite, 
in view of Lemma \ref{lem_numberOfSectionsDoesNotIncrease}, there exists some time $T>0$ when the number of sections becomes fixed and no longer changes. 
The sections may still expand or shrink though, preventing equilibration.
However, one can show that once a section expands from a certain direction, say left (i.e., in the descending order of the agents' labels), then it may no longer shrink from left.
Namely, if the left border of a section ``moves'' left after time $T$, it never ``moves'' right in the future. 
This idea is rigorously captured in the following lemma. 
For every time $t\geq T$, there is the same number of sections which we label as $1,2,\ldots,S$ from left to right, that is in the ascending order of their left borders. 
Denote by $L_s(t)$ and $R_s(t)$ the left and right borders of section $s$ at time $t\geq T$.
Given a sequence of consecutive agents $p,p+1,\ldots,q$, where $q\geq p$, denote their strategies by $\x_{(p,p+1,\ldots,q)}$.
\begin{lemma}\label{lem_leftBordersOnlyMoveLeft}
    Consider the time $T$ when the number of sections in the linear graph is fixed.
    Then for every section $s$ and any time $t_1\geq T$,
    \begin{gather} 
        \scalebox{.99}{$
        L_s(t_1+1) = L_s(t_1) - 1 
        \Rightarrow
        \forall t\geq t_1\ L_s(t+1)\! \leq\! L_s(t),$} \label{eq_lem_pathEq_statement1}\\
        \scalebox{.99}{$
         R_s(t_1+1) = R_s(t_1) + 1 
        \Rightarrow
         \forall t\geq t_1\
        R_s(t+1)\! \geq\! R_s(t).$} \nonumber
    \end{gather}
\end{lemma}
\begin{proof}
    We prove the first equation by contradiction; the proof of the second equation is similar. 
    Assume the contrary and let $t_3>t_1$ be the first time \eqref{eq_lem_pathEq_statement1} is violated, i.e., 
    $
        L_s(t_3+1) = L_s(t_3) + 1.
    $ 
    Let $t_2\in[t_1,t_3-1]$ be the last time that the left border of $s$ decreased, i.e., 
    $
        L_s(t_2+1) = L_s(t_2) - 1.
    $ 
    Let agent $i$ be the left border of section $s$ at time $t_2$, i.e., $i = L_s(t_2)$. 
    Then
    \begin{gather} 
        L_s(t_2) = L_s(t_3+1) = i, \label{eq_prop_pathEq_1}\\
        L_s(t) = i-1 \quad \forall t \in [t_2+1, t_3],  \label{eq_prop_pathEq_2} 
    \end{gather}
    Without loss of generality, assume that $x_i(t_2) = \B$.
    It is straightforward to show that if the agents of section $s$ play a strategy, say $\B$, at time $T$, then the agents of section $s$ will play $\B$ at every future time step as well. 
    Therefore, since agent $i$ is the left border of section $s$ at time $t_2$ and plays $\B$ at $t_2$, it follows that all the agents in section $s$ play $\B$ at every time $t\geq T$.
    Thus, in view of \eqref{eq_prop_pathEq_1} to \eqref{eq_prop_pathEq_2},
    \begin{align*}
        \x_{(i-2,i-1,i)}(t_2) & = (\A,\A,\B), \\
        \x_{(i-2,i-1,i)}(t_2+1) & = (\A,\B,\B), \\
        \x_{(i-2,i-1,i)}(t) & = (\A,\B,*) \quad \forall t\in[t_2+1,t_3-1],\\
        \x_{(i-2,i-1,i)}(t_3) & = (\A,\B,\B), \\
        \x_{(i-2,i-1,i)}(t_3+1) & = (\A,\A,\B).
    \end{align*}
    The reason why $x_{i-2}(t_2) = \A$ is that otherwise a section would be removed at $t_2+1$, which is impossible as the number of sections is assumed to be fixed after time $T$. 
    Similarly, $x_{i-2}(t_3) = \A$ as otherwise a new section would be generated at time $t_3+1$.
    
    Now we show that the two switches of strategies of agent $i-1$ at times $t_2+1$ and $t_3+1$ are in conflict. 
    Note that at both times $t_2$ and $t_3$ agent $i$ plays $\B$ but has at time $t_2$ at most and at time $t_3$ at least one other $\B$-playing neighbor.
    So as the game is coordinating, i.e., in view of \eqref{coordinationPayoffMatrix}, 
    $
        u_{i}(t_3) \geq u_{i}(t_2).
    $    
    We reach a contradiction in view of Lemma \ref{lem_impossible} and by letting $T=t_2$ and $T'=t_3$.
\end{proof}

We are ready to prove the equilibration of linear graphs.
Consider a section $s$ at time $T$. 
We say that \emph{the left border of section $s$ moves left at time $t\geq T+1$} if $L_s(t) = L_s(t-1) - 1$ and \emph{moves right} if  $L_s(t) = L_s(t-1) + 1$. 
Similarly, the movement of the right border is defined.
\begin{proposition}\label{prop_equlibrationOfPathStructuredPopulations}
    A linear graph equilibrates under the coordinating best-response and imitation dynamics with an arbitrary persistent activation sequence.
\end{proposition}
\begin{proof}
    Consider some section $s$ at time $T$ when the number of sections is fixed. 
    If the left border of section $s$ moves left at any future time, then it can only move left afterward according to Lemma \ref{lem_leftBordersOnlyMoveLeft}.
    Since the linear graph is constrained from left by agent $1$, the left border of section $s$ will be fixed at some time. 
    Similarly, the right border will be fixed if it moves right at some point.
    So if the left border moves left at some time and the right border moves right, then the borders of section $s$ will be fixed for all future times. 
    
    Now if any of the borders, say right, becomes fixed but the left one only moves right after time $T$, then also the left border becomes fixed at some point as it is bounded from right by the right border (cannot pass it).
    On the other hand, if the right border only moves left after time $T$ and the left only moves right, again the two will become fixed as they cannot pass each other. 
    Therefore, the borders of section $s$ will become fixed at some finite time. 
    Since $s$ was an arbitrary section, it holds that at some finite time, the borders of every section becomes fixed. 
    This implies equilibration as the activation sequence is persistent.
\end{proof}

\subsection{Extension to rings}
A network $\G = (\N,\E)$ with edge set $\E=\{\{1,2\},\dots,\{n,n+1\},\{n+1,1\}\}$ is a \emph{ring}.
\begin{proposition}[Rings]
 A ring network equilibrates under the coordinating best-response and imitation dynamics with an arbitrary persistent activation sequence.
\end{proposition}
\begin{proof}
    Following the same arguments used for the proof of the equilibration of a linear graph, it can be shown that the number of sections in a ring may not increase, and hence will become fixed at some time $T$, and that if the right border of a section moves right at some time, it may never move left afterwards, and vice versa.
    So the only possibility for the non-equilibration of a ring is that both borders of some section $s$ only and infinitely often move right or only and infinitely often move left. 
    Consider the second case, i.e., moving left.
    Then for every agent $i$ in the ring, there exists a time when it belongs to the section $s$ and a time when it does not belong to the section. 
    Hence, it will undergo the switches from $\x_{(i-1,i,i+1)} = (\A,\A,\B)$ to $(\A,\B,\B)$ and from $(\B,\B,\A)$ to $(\B,\A,\A)$. 
    So the agent decides differently at two states with the same number of $\A$ and $\B$-playing neighbors. 
    Thus, in view of \eqref{bestRespondersUpdateRule}, agent $i$, and hence, the whole ring are imitators.
    In view of the convergence result in \cite[Theorem 1]{ramazi2022lower} for arbitrary networks of all coordinating imitators, the proof is complete.
\end{proof}
\subsection{Starlikes} 
We now proceed to a more general network: The \emph{starlike} \cite{omidi2007starlike}, that is a tree with at most one branching agent.
Define a \emph{branch} as a linear graph that begins from a neighbour of the branching agent and ends with a leaf but does not contain the branching agent itself (Figure~\ref{fig:sparse1}-a).
 \begin{figure}[h]
        \centering
        \includegraphics[width=0.25\textwidth]{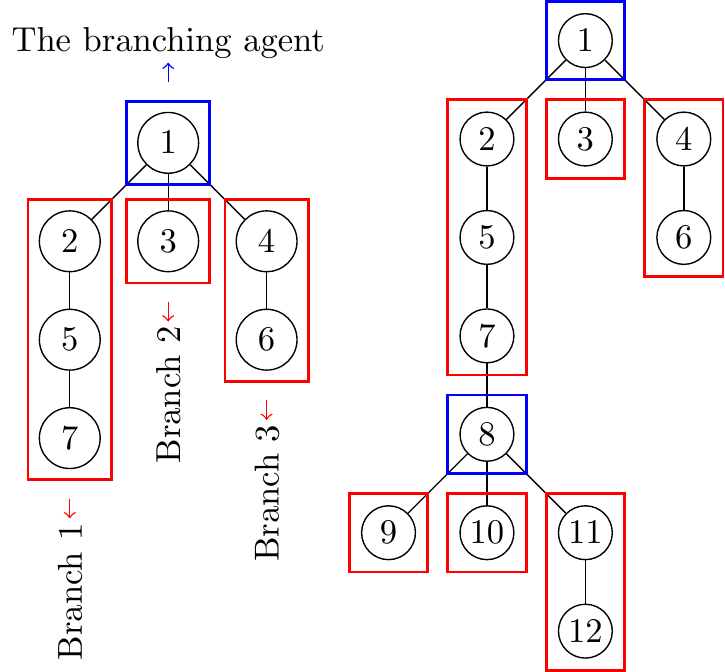}
        \caption{\small
        \textbf{a) Starlike graph.}
        The agent on the top is the branching agent.
    The graph has three branches.
        \textbf{b) a sparse-tree population.} 
        Each red section demonstrates a line in the population. Blue agents are branching agents of the population.}
        \label{fig:sparse1}
\end{figure}
\Cref{def_eventuallyPeriodicActivationSequence} and \Cref{lem_periodic_activation_seq} are for general decision-making dynamics but are framed here according to \Cref{sec_model}.
\Cref{def_eventuallyPeriodicActivationSequence} is based on the notion of \emph{eventually periodic sequences} \cite{braberman2013formal} and \Cref{lem_periodic_activation_seq} follows standard induction arguments.
\begin{definition}[Eventually periodic] \label{def_eventuallyPeriodicActivationSequence}
    The coordinating imitation and best-response dynamics are \emph{eventually periodic} under the activation sequence $\langle a_t\rangle_{t=0}^{\infty}$ if both the activation sequence and the resulting state $\x$ become periodic after some finite time $t_0$, i.e., 
    \begin{equation*}
        \exists T,t_0\in \mathbb{N}\
        \forall t \in \mathbb{N}\
        [a_{t_0+t+T} = a_{t_0+t}, \x_{t_0+t+T}=\x_{t_0+t}],
    \end{equation*}
    where $T$ is the \emph{periodicity after time $t_0$}, and time interval $[t_0,\infty)$ is the \emph{periodic interval}.    
    The activation sequence $\langle a_t\rangle_{t=0}^{\infty}$ is called an \emph{eventually periodic activation sequence}.
\end{definition}
\begin{lemma}\label{lem_periodic_activation_seq}
    If the coordinating imitation and best-response dynamics do not equilibrate under some persistent activation sequence, then there also exists a persistent eventually periodic activation sequence under which the population does not equilibrate. 
\end{lemma}
\begin{lemma}   \label{lem_starlike_branchingAgent}
    A starlike network equilibrates under the coordinating best-response and imitation dynamics with an arbitrary persistent activation sequence, if the branching agent does not switch strategies infinitely many times.
\end{lemma}
\begin{proof}
   It is straightforward to show that Lemmas \ref{lem_numberOfSectionsDoesNotIncrease} and \ref{lem_leftBordersOnlyMoveLeft} and hence proposition \ref{prop_equlibrationOfPathStructuredPopulations} hold if any of the end nodes in a linear graph fix their strategies. 
   Consequently, every branch together with the branching agent forms a linear graph in the starlike that will equilibrate, leading to the equilibration of the whole starlike.
\end{proof} 

The equilibration of starlike networks is established in Proposition \ref{prop_starlike}. 
The idea of the proof is to focus on the times when the branching agent has the maximum number of same strategy, say $\B$-playing, neighbors. 
The moment one of these neighbors, referred to as the ``special agent'' switches, the number of sections in the branch containing this special agent, referred to as the ``special branch'', will decrease, and this decrement will never be compensated in the future. 
So the number of sections in the the spacial branch is an energy-like function (see \eqref{prop_starlike_eq1}).
If before any of the neighbors switch, the branching agent itself switches, then the branching agent must be an imitator and the neighbor with the maximum utility will serve as the special agent.
Given a linear graph $P$, denote the number of sections in $P$ by $n(P)$, and more specifically by $n(P,t)$ to denote the number at time $t$. 
\begin{proposition}[Starlike]   \label{prop_starlike}
    A starlike network equilibrates under the coordinating best-response and imitation dynamics with an arbitrary persistent activation sequence.
\end{proposition}
\begin{proof}
    We prove by contradiction. 
    By assuming the contrary, Lemma \ref{lem_periodic_activation_seq} implies the existence of a persistent eventually periodic activation sequence denoted by $\langle b_t\rangle_{t=0}^{\infty}$ with periodic interval $[t_0,\infty)$.
    The branching agent, say $i$, switches strategies under $\langle b_t\rangle_{t=0}^{\infty}$ infinitely often; otherwise, the network equilibrates due to Lemma \ref{lem_starlike_branchingAgent}.
    Let $t_{\B}\geq t_0$ be the first time agent $i$ plays $\B$ and has the maximum number of $\B$-playing neighbours during $[t_0,\infty)$.
    Denote by $t_1$ the first moment after $t_{\B}$ when either agent $i$ or one of its neighbors switches strategies, resulting in the following two cases:

    \emph{Case 1: An agent $i$’s neighbour, say agent $j$, changes its strategy at time $t_1$.}
    Consider the branch $P$ (referred to as the ``special'' branch) including agent $j$ (referred to as the ``special'' agent). 
    Denote the times when agent $j$ switches strategies after $t_1$ by $t_2,t_3\ldots$.
    Out of these time steps,
    let $\langle t^\B_k\rangle_{k=1}^\infty\subset 
    \langle t_k\rangle_{k=1}^\infty$ be those time steps such that agent $i$ had its maximum number of $\B$-playing neighbors at each time $t^\B_k-1$.
    Clearly, $t^\B_1 = t_1$.
    We show that 
    \begin{equation}    \label{prop_starlike_eq1}
        \forall k\geq1\quad  n(P,t^\B_{k+1}) - n(P,t^\B_k) \leq -1.     
    \end{equation}
    At every time $t^\B_k$ agent $j$ switches from $\B$ to $\A$; otherwise, agent $i$ will have more $\B$-neighbors at time $t^\B_k$ compared to $t_{\B}$.
    Thus, as the dynamics are coordinating, agent $j$ has at least one $\A$-playing neighbour at $t^\B_k$, who is not agent $i$.
    Because the network is starlike, agent $j$ has at most two neighbors, so it has exactly one other neighbor, say agent $k$, who plays $\A$ at $t^\B_k$.
    So $n(P)$ reduces by one at time $t^\B_k$. 
    In view of Lemma \ref{lem_SectionReduction_in_general_paths}, $n(P)$ does not increase if any agent other than $j$ switches strategies.
    So $n(P)$ may increase in the future, only at times when agent $j$ switches strategies, i.e., $t_{k+1},t_{k+2},\ldots$.
    We show that $n(P)$ decreases at each time $T=t_{2r}$ for an arbitrary $r\in\mathbb{N}$.
    Agent $j$ switches from $\A$ to $\B$ at time $t_{2r}$.
    If neighbor $k$ plays $\A$ at time $T-1$, then agent $i$ plays $\B$ at the same time; otherwise, agent $j$ does not tend to switch.
    So $\x_{(k,j,i)}(T-1) = (\A,\A,\B)$.
    Having the maximum number of $\B$-playing neighbors at time $t_1$, 
    agent $i$'s utility at time $T-1$ is no more than at time $t_1-1$: 
    $
        u_i(T-1) \leq u_i(t_1-1).
    $    
    But this is impossible according to Lemma \ref{lem_impossible}. 
    So neighbor $k$ plays $\B$ at time $T-1$.
    Then $n(P)$ reduces by the switch of agent $j$ at time $T$.
%
    On the other hand, $n(P)$ may increase by at most one at each time $t_{2r+1}, r\in\mathbb{N}$.
    Therefore, there is no finite time $T'>t^\B_k$ when $n(P)$ equals its value at $t^\B_k$, proving \eqref{prop_starlike_eq1}, a contradiction.

    \emph{Case 2: Agent $i$ switches from $\B$ to $\A$ at time $t_1$.}
    There exists time $t_2>t_1$ when agent $i$ tends to switch back to $\B$.
    However, the number of agent $i$'s $\B$-playing neighbors is maximized at time $t_1$, when it switched to $\A$. 
    Hence, because of the coordinating dynamics, agent $i$ is an imitator.
    
    Denote by $\langle T_r\rangle_{r=0}^\infty$ the time steps after $t_0$ that agent $i$ changes its strategy, and let $\langle a_r\rangle_{r=0}^\infty$ be the corresponding neighbors imitated by agent $i$.
    Let $a_j$ be an agent among $\langle a_r\rangle_{r=0}^\infty$ with the maximum utility, i.e., 
    $
        a_j = \arg\max_{r} u_{a_r}.
    $     
    So the maximum utility among the agent $i$'s neighbors was earned by agent $a_j$ at time $T_j$.
    Consider the branch $P$ including agent $a_j$.
    We show that the number of sections in $P$ decreases at least once after time $T_j$ but never increases afterwards, which is in contradiction with $T_j$ belonging to the periodic interval of the activation sequence. 
    
    First, we prove the following:
    \emph{Statement 1. At any time $T_r$, $r\geq0$, when agent $i$ switches to $\A$, agent $a_j$ must also play $\A$.}
    At time $T_j-1$, agent $a_j$ plays $\B$ and has at most one $\B$-playing neighbor.
    At time $T_r-1$, agent $i$ plays $\B$, so agent $a_j$ has at least one $\B$-playing neighbor. 
    So if agent $a_j$ plays $\B$ at time $T_r$, it earns no less than at time $T_j$ because of the coordinating dynamics, i.e., 
    $
        u_{a_j}(T_r) \geq u_{a_j}(T_j).
    $
    Hence, according to the definition of $a_j$, agent $a_j$ is a maximum earner at time $T_r$. 
    Since agent $i$ does switch at time $T_r$, it has to switch to the strategy of agent $a_j$ according to \eqref{imitatorsUpdateRule}.
    This is, however, impossible since both agents $a_j$ and $i$ play the same strategy $\B$ at time $T_r$. 
    This proves Statement 1. 

    
    Next, we list and investigate the possible strategy states for the pair $(i,a_j)$ starting from time $T_j$:

    \emph{Case 2.1. $\x_{(i,a_j)}(t) = (\B, \B)$.}
    Then at the next time $T_r\geq t$ when agent $i$ changes strategies, it switches to $\A$. 
    Hence, according to Statement 1, $\x_{(i,a_j)}(T_r-1) = (\B, \A)$.
    So according to Lemma \ref{lem_SectionReduction_in_general_paths}, $n(P)$ reduces by at least 1 during $[t,T_r-1]$ as agent $i$ does not switch strategies in this interval.
    We reach Case 2.2 at time $T_r$ as  $\x_{(i,a_j)}(T_r) = (\A, \A)$.
    
    \emph{Case 2.2. $\x_{(i,a_j)}(t) = (\A, \A)$.}
    Then at the next time $T_s\geq t$ when agent $i$ changes strategies, it switches to $\B$.
    Now if agent $a_j$ plays $\B$ at time $T_s-1$, we have $\x_{(i,a_j)}(t) = (\A, \B)$. 
    So again according to Lemma \ref{lem_SectionReduction_in_general_paths}, $n(P)$ reduces by at least 1 during $[t,T_s-1]$.
    We reach Case 1 at time $T_s$ as  $\x_{(i,a_j)}(T_s) = (\B, \B)$.
    Now if agent $a_j$ plays $\A$ at time $T_s-1$, we have $\x_{(i,a_j)}(t) = (\A, \A)$. 
    So according to Lemma \ref{lem_SectionReduction_in_general_paths}, $n(P)$ may not increase during $[t,T_s-1]$.
    We reach Case 2.3 at time $T_s$ as  $\x_{(i,a_j)}(T_s) = (\B, \A)$.
    
    \emph{Case 2.3. $\x_{(i,a_j)}(t) = (\B, \A)$.}
    Then at the next time $T_p\geq t$ when agent $i$ changes strategies, it switches to $\A$.
    Hence, according to Statement 1, $\x_{(i,a_j)}(T_p-1) = (\B, \A)$ which is the same as the state at time $t$ in this case.
    So according to Lemma \ref{lem_SectionReduction_in_general_paths}, $n(P)$ does not increase during $[t,T_p-1]$.
    We reach Case 2.2 at time $T_p$ as  $\x_{(i,a_j)}(T_p) = (\A, \A)$.
    
    At time $T_j$, the strategy state $\x_{(i,a_j)}$ matches Case 1, where $n(P)$ reduces. 
    The proof is complete since it does not increase afterwards in any of the above cases. 
\end{proof}

\subsection{Sparse-trees}
We are ready to prove \Cref{th_sparseThree}. 
The idea is to show that the ``special branches'' of two branching nodes will intersect, resulting in the so-called ``golden branch'' (\Cref{fig:sparse1}) which is guaranteed to equilibrate.

\begin{proof}[Proof of \Cref{th_sparseThree}]
    Equilibration of starlike networks were shown in Proposition \ref{prop_starlike}.
    So here we consider the case with at least two branching agents.
    We prove by contradiction and consider a persistent eventually periodic activation sequence denoted by $\langle b_t\rangle_{t=0}^{\infty}$ with periodic interval $[t_0,\infty)$. 
    Similar to the proof of lemma \ref{lem_starlike_branchingAgent}, it can be shown that at least one branching agent changes its strategy during the periodic interval of the oscillation. 
    We refer to the agents who change their strategy during $[t_0,\infty)$ a \emph{settling} agent and otherwise \emph{unsettling}.
    For each unsettling agent $i$, denote its special branch defined in the proof of Proposition \ref{prop_starlike} by $P_i$. 
    Equilibration can be shown using Lemma \ref{lem_starlike_branchingAgent} when there is no unsettling branching agent and similar to 
    Proposition \ref{prop_starlike} when the special branches of no two branching agents overlap (no golden branch). 
    So consider the case where there are two branching agents with the corresponding special neighbors $i$ and $j$, and whose special branches intersect, denoted by $P$.
    In view of Lemma \ref{lem_SectionReduction_in_general_paths}, $n(P)$ increases only at the time steps when either agent $i$ or $j$ switches. 
    On the other hand, for both Case 1 and 2 in Proposition \ref{prop_starlike}, it is guaranteed that there exists some infinite time series $\langle t^i_k\rangle_k^\infty$ (when agent $i$ switches) such that
    $
        n(P,t^i_{k+1}) - n(P,t^i_{k}) \leq -1
    $ for all $k\geq0$,
    and 
    a some time series $\langle t^j_k\rangle_k^\infty$ (when agent $j$ switches) such that
    $
        n(P,t^j_{k+1}) - n(P,t^j_{k}) \leq -1
    $
    for all $k\geq0$.
    This is a contradiction as then $n(P)$ is unbounded. 
\end{proof}

\section{Conclusion}
We showed that every sparse tree network of coordinating heterogeneous imitators and best-responders equilibrates under any persistent activation sequence.
This implies that neither the heterogeneity in the agents' perceptions of the coordination game (i.e., different payoff matrices), nor the order the agents become active can cause fluctuations in the mixed population, at least when their connections are as sparse as a sparse-tree.
Whether dense trees or general graphs equilibrate under every activation sequence remains an open problem. 
For the proof, we introduced the number of sections in a linear graph as a potential function and generalized it to the starlike and then sparse tree networks.
The potential functions may be tested in other decision making dynamics. 
For example, it is expected for the number of sections to increase and eventually become fixed in anti-coordination games under certain conditions \cite{fates2006asynchronous}.


\section*{Appendix}
\begin{lemma}   \label{lem_impossible}
    Consider a network governed by the coordinating best-response and imitation dynamics with an arbitrary activation sequence.
    Assume that the network includes neighboring agents $p-1$ and $p$, each of degree two, and denote the other neighbor of agent $p$ by $p+1$.
    If there exists some time $T\geq0$ when agent $p$ tends to switch strategies and 
    $\x_{(p-1,p,p+1)}(T) = (\A,\A,\B),$ 
    then agent $p$ does not tend to switch strategies at any time $T'$ when
    $\x_{(p-1,p,p+1)}(T') = (\A,\B,\B)$
    and when agent $p+1$ earns non-less, i.e., $u_{p+1}(T')\geq u_{p+1}(T)$.
\end{lemma}
\tb{\begin{proof}
    Should agent $p$ be a best-responder, its tendency to switch strategies at time $T$ implies that it tends to play $\B$ if at least one of its neighbors plays $\B$ in view of \eqref{bestRespondersUpdateRule} and \eqref{coordinationPayoffMatrix}.
    Therefore, agent $p$ also tends to play $\B$ at time $T'$ since it has a $\B$-playing neighbor.  
    So consider the case where agent $p$ is an imitator. 
    At time $T$ agent $p$ tends to imitate agent $p+1$ who plays $\B$ and has at most one other $\B$-playing neighbor.
    So agent $p+1$ earns more than agent $p-1$, i.e.,
    $        
        u_{p+1}(T) > u_{p-1}(T). 
    $
    We know that agent $p+1$ earns at time $T$ no more than at time $T'$, i.e.,
   $
        u_{p+1}(T) \leq u_{p+1}(T').
    $
    Moreover, at both times $T$ and $T'$, agent $p-1$ plays $\A$ but has at time $T$ at least and at time $T'$ at most one other $\A$-playing neighbor, implying
    $
        u_{p-1}(T) \geq u_{p-1}(T').
    $ 
    These inequalities result in 
    $
    u_{p+1}\left(T'\right)>u_{p-1}\left(T'\right),
    $
    which implies the imitator agent, $p$, does not tend to change its strategy to agent $p-1$'s. 
    completing the proof.    
\end{proof}}

We say that a network \emph{admits a linear graph} $(1,2,\ldots,m)$ if there is a link between node $i$ and $i+1$ for all $i=1,\ldots,m-1$ and the degree of every node $2,\ldots,m-1$ is two. 
We refer to $(2,\ldots,m-1)$ as the \emph{interior} of the linear graph. 
\begin{lemma}\label{lem_SectionReduction_in_general_paths} 
    Consider a network admitting the linear graph $(1,\ldots,m)$ governed by the coordinating best-response and imitation dynamics. 
    Then the number of sections in the interior of the linear graph does not increase if each of the ending agents $1$ and \textcolor{blue}{$m$} either are a leaf or its strategy does not change under the activation sequence. 
\end{lemma}
\tb{\begin{proof}
    The proof follows \Cref{lem_numberOfSectionsDoesNotIncrease} with an activation sequence that does not activate an ending agent whose strategy is fixed. 
\end{proof}}

\bibliography{main}

\end{document}